\documentclass[12pt]{amsart}
\usepackage{amsmath,amsthm,amsfonts,amssymb,amscd}
\usepackage[all]{xy}

\newtheorem{thm}{Theorem}[section]
\newtheorem{cor}[thm]{Corollary}
\newtheorem{prop}[thm]{Proposition}
\newtheorem{lemma}[thm]{Lemma}

\theoremstyle{plain}

\theoremstyle{definition}

\newtheorem{Emp}[thm]{}

\numberwithin{equation}{section}
\renewcommand{\1}{\mathbf{1}}

\newcommand{\R}{\mathbb{R}}
\newcommand{\Z}{\mathbb{Z}}
\newcommand{\N}{\mathbb{N}}
\newcommand{\GL}{\mathrm{GL}}
\newcommand{\SL}{\mathrm{SL}}

\newcommand{\der}{\mathrm{der}}

\newcommand{\Frob}{\mathrm{Frob}}

\newcommand{\Conv}{\mathrm{Conv}}

\newcommand{\Qlb}{\bar{\mathbb{Q}}_\ell}

\newcommand{\isom}{\overset{\thicksim}{\to}}
\newcommand{\cC}{\mathcal{C}}

\newcommand{\la}{\lambda}
\newcommand{\om}{\omega}
\newcommand{\al}{\alpha}
\newcommand{\Gm}{\Gamma}
\newcommand{\gm}{\gamma}

\newcommand{\Tr}{\operatorname{Tr}}
\newcommand{\Spec}{\operatorname{Spec}}
\newcommand{\Rep}{\operatorname{Rep}}
\newcommand{\Aut}{\operatorname{Aut}}
\newcommand{\Hom}{\operatorname{Hom}}

\newcommand{\Vect}{\operatorname{Vec_K}}
\newcommand{\wt}{\widetilde}

\begin{document}

\title[Tannakian Formalism and Langlands Conjectures]
{The Tannakian Formalism and the Langlands Conjectures}

\author{David Kazhdan}
\email{kazhdan@math.huji.ac.il}
\address{Einstein Institute of Mathematics\\
    Hebrew University \\
    Givat Ram, Jerusalem 91904\\
    Israel}

\author{Michael Larsen}
\email{mjlarsen@indiana.edu}
\address{Department of Mathematics\\
    Indiana University \\
    Bloomington, IN 47405\\
    U.S.A.}

\author{Yakov Varshavsky}
\email{vyakov@math.huji.ac.il}
\address{Einstein Institute of Mathematics\\
    Hebrew University \\
    Givat Ram, Jerusalem 91904\\
    Israel}

\thanks{DK was partially supported by ISF grant 1438/06.
ML was partially supported by NSF Grant DMS-0800705.
YV was partially supported by ISF grant 598/09.}

\begin{abstract}
Let $H$ be a connected reductive group over an algebraically closed field of characteristic zero, and let $\Gm$ be an abstract group.
In this note we show that  every homomorphism of  Grothendieck semirings $\phi:K_0^+[\Gm]\to K_0^+[H]$, which maps irreducible
representations to irreducible, comes from a group homomorphism  $\rho:\Gm\to H(K)$. We also connect this result with the Langlands conjectures.

\end{abstract}

\maketitle

\section*{Introduction}
Let $F$ be a global function field, $\Gm_F$ the absolute Galois group of $F$,  $G$ a split connected reductive group over $F$,   $l$  a prime number different from the characteristic of $F$, and  $\hat G={}^LG^0$ the connected Langlands dual group over $\Qlb$.

Recall that a weak  Langlands conjecture asserts that for every $(\pi,\om)$, where $\pi$ is
an automorphic representation $\pi$ of $G$, whose central character is of finite order, and $\om$ is a representation of $\hat G$, there exists  a unique semisimple $\ell$-adic representation $\rho_{\pi,\om}$ of $\Gm_F$, whose
$L^S$-function is equal to the $L^S$-function of $(\pi,\om)$.

 Moreover, a strong Langlands conjecture asserts that there exists a $\hat G$-valued $\ell$-adic  representation
 $\rho_{\pi}:\Gm_F\to \hat G(\Qlb)$ (not unique in general) such that the composition $\om\circ \rho_{\pi}$ is isomorphic to $\rho_{\pi,\om}$ for each representation
 $\om$.

The main result of this note implies that in some cases the strong Langlands conjecture follows from the weak one. More specifically we show the existence
 of $\rho_{\pi}$  in the case when $\rho_{\pi,\om}$ is irreducible for each irreducible representation $\om$. Moreover, in this case, $\rho_{\pi}$ is unique up to conjugation, and the Zariski closure of its image contains the derived group of $G$.

Our result is a corollary of the following variant of the Tannakian formalism. Let $H$ be a connected reductive group over an algebraically closed field $K$ of characteristic zero, and let $\Gm$ be an abstract group. Then every homomorphism  of groups $\rho:\Gm\to H(K)$ induces a homomorphism of  Grothendieck
semirings $\rho^*:K_0^+[\Gm]\to K_0^+[H]$.  In this note we show a partial converse of this assertion. Namely, we show that every homomorphism of  Grothendieck
semirings $\phi:K_0^+[\Gm]\to K_0^+[H]$, which maps irreducible representations to irreducibles, comes from a group homomorphism  $\rho:\Gm\to H(K)$.
In particular, we show that a connected reductive group is determined by its Grothendieck semiring.

This note was inspired by a combination of a work in progress \cite{KV}, where it is shown that a weak Langlands conjecture holds in many cases, and a work
\cite{LP}, which indicates that one does not need the full Tannakian structure in order to reconstruct a connected reductive group.

We would like to acknowledge the contribution of Pavel Etingof, whose suggestions led to substantial simplification and conceptual clarification of this paper.

Part of the work was done while DK and YV visited the University of Chicago.
The rest of the work was done while YV visited Indiana University.
We thank both these institutions for stimulating atmosphere and financial support.
We would also like to thank Michael Mueger for calling our attention to two previous proofs of Theorem~\ref{kgroups} below.

\section{Main results}
Let $K$ be an algebraically closed field of  characteristic zero.

\begin{Emp}
(a) For every  algebraic group $G$ over $K$, we denote by $K_0^+[G]$ the Grothendieck semiring of the category of rational representations
of $G$.

In other words, $K_0^+[G]$ is the set of equivalence classes of finite-dimensional semisimple representations of $G$. For every representation
$\om$ of $G$, we denote by $[\om]$ its class (or more precisely, the class of its semisimplification) in $K_0^+[G]$.  For every pair of semisimple representations $\om_1$ and $\om_2$ of $G$, we have
$[\om_1]+[\om_2]=[\om_1\oplus\om_2]$ and $[\om_1]\cdot[\om_2]=[\om_1\otimes\om_2]$.

(b) Note that a representation $\om$ of $G$ is irreducible, if and only if its class $[\om]\in K_0^+[G]$ is irreducible, that is, it cannot be realized as a nontrivial sum
$[\om_1]+[\om_2]$ of elements of  $K_0^+[G]$.

(c) Every homomorphism $\rho:G\to H$ of algebraic group over $K$ gives rise to the  homomorphism
$\rho^*:K_0^+[H]\to K_0^+[G]$ of semirings, where $\rho^*([\om]):=[\om\circ\rho]$.

\end{Emp}
The following result asserts that each connected reductive group is determined by its Grothendieck semiring.
\begin{thm}
\label{kgroups}
Let $G$ and $H$ be two connected reductive groups over $K$, and let $\phi:K_0^+[G]\isom K_0^+[H]$ be an isomorphism of semirings.

Then there exists an isomorphism $\rho\colon H \isom G$ such that $\rho^*=\phi$. Moreover, $\rho$ is unique up to conjugation.
\end{thm}

\begin{Emp}
{\bf Remark.} Note, by comparison, that if $G$ is connected, semisimple, and simply connected, then the Grothendieck \emph{ring} $K_0[G]$ is isomorphic to
$\Z[x_1,\ldots,x_r]$, where $r$ is the rank of $G$.  Thus, for such groups, $K_0[G]$ encodes only the rank.
\end{Emp}

\begin{Emp}
Now let  $\Gm$ be an abstract group, and let $K_0^+[\Gm]$ be the Grothendieck semiring of the category of finite-dimensional representations
of $\Gm$ over $K$. Every group homomorphism $\rho:\Gm\to G(K)$ gives rise to the homomorphism
 $\rho^*:K_0^+[G]\to K_0^+[\Gm]$ of Grothendieck semirings.
 \end{Emp}

We have  the following version of the Tannakian formalism.

\begin{thm}
\label{main}
Let $\phi:K_0^+[G]\to K_0^+[\Gm]$ be a homomorphism of semirings which maps irreducible elements to irreducibles.

Then  there exists a  homomorphism
$\rho\colon \Gamma \to G(K)$ such that  $\rho^*=\phi$.
Moreover,  the Zariski closure of the image of each such $\rho$ contains $G^{\der}$, and $\rho$ is unique up to conjugation.
\end{thm}
\begin{Emp}
\label{remmain}
{\bf Remarks.}

(a) Conversely, let  $\rho\colon \Gamma \to G(K)$ be a homomorphism such that  the Zariski closure of $\rho(\Gamma)$ contains $G^{\der}$.
Then the homomorphism $\rho^*:K_0^+[G]\to K_0^+[\Gm]$ maps irreducible elements to irreducibles.

(b) The result fails completely if one does not assume that $\phi$ maps irreducible elements to irreducible.

Indeed, let $G$ be $\SL_2$, and let $\Gm$ be the group with one element. In this case,  for each integer
$k\geq 2$ there exists a (unique) homomorphism of semirings
$$\phi_k: K_0^+[\SL_2]\to K_0^+[\Gm]=\Z_{\geq 0},$$
which maps the standard representation of
$\SL_2$ to $k\in\Z_{\geq 0}$.  Only $\phi_2$ corresponds to a (unique) homomorphism $\Gm\to \SL_2(K)$.
 \end{Emp}


\begin{Emp}
\label{cg}
{\bf Chevalley space.}
(a) Let $c_G:=\Spec K[G]^G$ be the  Chevalley space of $G$. For every representation $\om$ of $G$, its trace $\Tr_{\om}\in K[G]^G\subset K[G]$ is a
regular function on $c_{G}$.

(b) Let $\chi_G\colon G\to c_G$ be the canonical projection,
induced by the embedding $K[c_G]=K[G]^G\hookrightarrow K[G]$. Then  for each $g\in G$
and each representation $\om$ of $G$, we have an equality $\Tr_{\om}(\chi_G(g))=\Tr_{\om}(g)$.
\end{Emp}

The following result is a more explicit formulation of Theorem~\ref{main}.

\begin{cor}
\label{cmain}
Let $f:\Gm\to c_G(K)$ be a map of sets.

Suppose that  for every irreducible algebraic representation  $\om$ of $G$, there exists an
irreducible finite dimensional representation $\rho_\om$ of $\Gm$ over $K$ such that
\begin{equation}
\label{trace-condition}
\Tr_{\rho_\om}(\gm) = \Tr_\om(f(\gm))  \text{ for all }\gm\in\Gm.
\end{equation}
Then there exists a  homomorphism $\rho\colon \Gamma \to G(K)$ such that
\begin{equation}
\label{lift}
\chi_G(\rho(\gm)) = f(\gm ) \text{ for all }\gm\in\Gm.
\end{equation}
Moreover, the Zariski closure of $\rho(\Gamma)$ contains $G^{\der}$, and $\rho$ is unique up to conjugation.
\end{cor}
\begin{Emp}
{\bf Remark.}
Conversely, assume that there exists a homomorphism   $\rho\colon \Gamma \to G(K)$ satisfying (\ref{lift}) and such that
 the Zariski closure of $\rho(\Gamma)$ contains $G^{\der}$. Then for every irreducible representation $\om\colon G\to \GL_{n}$
 the representation
 $$\rho_{\om}:=\om\circ\rho:\Gm\to \GL_{n}(K)$$
 is irreducible  and satisfies  (\ref{trace-condition})
 (use \ref{cg} (b)).
 \end{Emp}

\begin{Emp}
{\bf Application to the Langlands conjectures.}

Let $F$ be a global field, $F^{\mathrm{sep}}$ a separable closure, $\Gamma_F = \mathrm{Gal}(F^{\mathrm{sep}}/F)$ the absolute Galois group, and
$\ell$ a prime number, different from the characteristic of $F$. Let $\hat G$ be a connected reductive group over $\Qlb$.

By an {\em $\ell$-adic} (resp. {\em $\hat G$-valued $\ell$-adic}) representation of $\Gm_F$, we mean a continuous homomorphism $\rho\colon \Gamma_F\to \GL_n(\Qlb)$ (resp. $\rho\colon \Gamma_F\to \hat G(\Qlb)$) which is unramified for almost all places of $F$.

There is a well-defined trace
$\Tr_{\rho}(\Frob_v)$ (resp.  $\Tr_{\om\circ\rho}(\Frob_v)$) for almost all places  $v$ of $F$ (resp. and all representations $\om$ of $\hat G$).
\end{Emp}

The following analogue of Corollary~\ref{cmain} has applications to Langlands conjectures.

\begin{cor}
\label{galrep}
Let $\hat G$ be a reductive group over $\Qlb$,  $\Sigma$ a cofinite subset of the places of $F$,  and $f\colon \Sigma\to c_{\hat G}(\Qlb)$  any map of sets.

Assume that for every irreducible algebraic representation $\om$ of $G$, there exists an irreducible
$\ell$-adic  representation $\rho_\om$ of $\Gamma_F$ such that
\begin{equation}
\label{frob-condition}
\Tr_{\rho_\om}(\Frob_v) = \Tr_{\om}(f(v))\text { for almost all } v\in \Sigma.
\end{equation}
Then there exists a $\hat G$-valued $\ell$-adic representation $\rho\colon \Gamma_K\to \hat G(\Qlb)$ such that
\begin{equation}
\label{lift1}
\chi_{\hat G}(\rho(\Frob_v)) = f(v)\text{  for almost all }v\in \Sigma.
\end{equation}
Moreover, the Zariski closure of $\rho(\Gamma_F)$ contains $\hat G^{\der}$,  and $\rho$ is unique up to conjugation.
\end{cor}

\section{Determining connected reductive group by its Grothendieck semiring}
In this section, we are going to prove Theorem~\ref{kgroups}.
Michael Mueger called our attention to the fact that at least two proofs of this theorem already exist in the literature: \cite{Mc} and \cite{Ha}.  Nevertheless, we feel that this new proof has 
merits (including brevity) which justify presenting it.

Let $G$ be a connected reductive group. We will fix a Borel subgroup $B\subset G$ and a maximal torus $T\subset B$.
Let $\al_1,\ldots,\al_r$ be the simple roots of $G$ with respect to $(B,T)$, and let $W$ be the Weyl group of $(G,T)$.

\begin{Emp}
\label{notation-weights}

(a) For each subset $X\subset X^*(T)$ we denote by $\Conv(X)$ the convex hull of
$X\subset X^*(T)\otimes\R$.

(b) For each dominant weight $\nu$ of $G$ we denote by $V_{\nu}$ the irreducible representation of $G$ with highest weight $\nu$.

(c) We define a partial order on $X^*(T)$ by the rule $\mu\leq\la$ if and only if $\la=\mu+\sum_{i=1}^r x_i\al_i$ and $x_i\geq 0$ for all $i$.

\end{Emp}

\begin{prop}
\label{order}
Let $\mu$ and $\la$ be two dominant weights of $G$. The following conditions are equivalent:

(a) $\mu\leq \la$;

(b) $\Conv(W\mu)\subset \Conv(W\lambda)$;

 (c) There exists a finite dimensional representation $V'$ of $G$ such that
for every $n$, every irreducible factor of $V_\mu^{\otimes n}$ is a factor of
$V_\lambda^{\otimes n}\otimes V'$.
\end{prop}

\begin{proof}

$(a)\implies (b)$.
Notice that since $\mu$ is dominant, we have $w\mu\leq\mu$ for all $w\in W$.
Therefore our assumption $\mu\leq\la$ implies that $w\mu\leq\la$ for all $w\in W$.
Thus our assertion follows from the following lemma.

\begin{lemma}
\label{worbit}
Let $\mu$ and $\la$ be two weights of $G$ such that  $w\mu\leq\la$ for all $w\in W$.
Then $\Conv(W\mu)\subset \Conv(W\lambda)$.
\end{lemma}

\begin{proof}
Suppose $\Conv(W\mu)$ is not contained in $\Conv(W\lambda)$.
Then there exists $w\in W$ such that $w\mu\not\in \Conv(W\lambda)$. As
$\Conv(W\lambda)$ is $W$-stable, it follows that $\mu\not\in \Conv(W\lambda)$, and hence also
$w\mu\notin \Conv(W\lambda)$  for all $w\in W$.

By the separation lemma, there exists  $\theta\in V^*$ such that
$\theta(\mu) > \theta(w\la)$ for all $w\in W$.
This is an open condition, so we may choose $\theta$ such that $\theta(\alpha_i)\neq 0$
for each $i=1,\ldots,r$. Replacing $\theta$ by $w\theta$ and $\mu$ by $w\mu$ for some
$w\in W$, we may assume in addition that $\theta(\alpha_i)>0$ for each $i=1,\ldots,r$.

By our assumption, $\mu=\la-\sum_{i=1}^r x_i\al_i$ with each $x_i\geq 0$. Therefore
$$\theta(\mu)=\theta(\la)-\sum_{i=1}^r x_i\theta(\al_i)\leq\theta(\la),$$
contradicting our assumption
$\theta(\mu) > \theta(\la)$.
\end{proof}

$(b)\implies (c)$.
We start with the following lemma.
\begin{lemma}
\label{convex}
Let $X$ be a finite subset of a finite dimensional Euclidean space $E$. Then there exists a compact subset $Y$ of $E$ such that
\begin{equation}
\label{quantize}
\Conv(nX) \subset Y+\underbrace{X+X+\cdots+X}_n
\end{equation}
for all positive integers $n$.
\end{lemma}
\begin{proof}
Let $m:= |X|$, and let $Y$ denote the ball of vectors of norm at most
$R := 2m \max_{x\in X} \Vert x\Vert.$
We claim that inclusion (\ref{quantize}) holds for this $Y$.

Let $X$ be the set $\{x_1,\ldots,x_m\}$. Then every vector in $\Conv(nX)$ is of the form
$$v := a_1 n x_1 + \cdots + a_m n x_m,$$
where the $a_i$ are non-negative and sum to $1$.  Let $b_i := \lfloor na_i\rfloor$ for $i\ge 2$
and $b_1 = n-(b_2+\cdots+b_m)$.  As $|b_i - a_i n| < 1$ for $i>1$, we have
$$|b_1 - a_1 n| = |n-(b_2+\cdots+b_m) + (a_2 n+\cdots+a_m n-n)| < m-1.$$
Thus,
$$\Vert(b_1 x_1 + \cdots + b_m x_m) - v \Vert \le \sum_{i=1}^m |b_i - a_i n| \Vert x_i\Vert
< R,$$
and of course $b_1 x_1+\cdots+b_m x_m$ belongs to the $n$-fold iterated sum of $X$.
\end{proof}

Now we return to the proof of the proposition. We assume that $\Conv(W \mu)\subset \Conv(W \lambda)$, let $X = W\lambda$, and fix a compact set $Y$ satisfying (\ref{quantize}).  Denote by $V'$  the direct sum of all representations $V_\nu$ where $\nu$ ranges over the dominant weights in $W Y$.

If $n$ is a positive integer, the highest weight $\chi$ of any irreducible factor of $V_\mu^{\otimes n}$ is a weight of
$V_\mu^{\otimes n}$. Therefore $\chi\leq n\mu$, hence by the implication $(a)\implies(b)$ shown above, $\chi$ is an element
of
\[
\Conv(Wn\mu)= n\Conv(W\mu)\subset n\Conv(W\la).
\]
By (\ref{quantize}), $\chi$ can be written as a sum of $n$ elements of $W\la$ and an element of $W Y$,
which is necessarily in the weight group.
Thus $\chi$ has the form $\sum_{i=1}^n w_i\la+w'\nu$ for some $w_1,\ldots,w_n,w'\in W$ and some highest weight $\nu$ of $V'$.

Using the conjecture of Parthasarthy, Ranga-Rao, and Varadarajan, proven by Kumar \cite{Ku}, we conclude that $V_{\chi}$
is an irreducible factor of $V_\lambda^{\otimes n}\otimes V_{\nu}$, hence also an irreducible factor of $V_\lambda^{\otimes n}\otimes V'$.

$(c)\implies(a)$.
Now suppose that there exists a finite dimensional representation $V'$ of $G$ such that
for every $n$, every irreducible factor of $V_\mu^{\otimes n}$ must be a factor of
$V_\lambda^{\otimes n}\otimes V'$ as well.  Then every weight of $V_\mu^{\otimes n}$
must be a weight of $V_\lambda^{\otimes n}\otimes V'$, and in particular, this is true for the weight
$n\mu$. Thus $n\mu=\la_n+\nu_n$ for some weights $\la_n$ of  $V_\lambda^{\otimes n}$ and $\nu_n\in V'$.

Note that $\la_n=n\la-\sum_{i=1}^r n_i\al_i$ for some $n_i\in\Z_{\geq 0}$. Therefore $n\mu$ is equal to
$n\la-\sum_{i=1}^r n_i\al_i+\nu_n$, hence for each $n\in\N$ we have an equality
\[
\la-\mu=\sum_{i=1}^r \frac{n_i}{n}\al_i-\frac{1}{n}\nu_n.
\]
Next we recall that the set of weights  of $V'$ is finite, so the expression $\frac{1}{n}\nu_n\in X^*(T)\otimes\R$
tends to zero when $n$ tends to infinity. Hence the difference $\la-\mu$ equals $\sum_{i=1}^r x_i\al_i$, where each $x_i=\lim_{n\to\infty}\frac{n_i}{n}$
is non-negative. This shows that $\mu\leq\la$.
\end{proof}

\begin{cor}
\label{root datum}
The root datum of $G$ can be reconstructed from the semiring $K_0^+[G]$.
\end{cor}

\begin{proof}
We divide our construction into steps as follows.

{\bf Step 1.} First we claim that a partially ordered set of dominant weights of $G$ can be reconstructed from the semiring $K_0^+[G]$. 

For this we note that the map $\mu\mapsto [V_{\mu}]$
gives a bijection between the set of dominant weights of $G$  and the set of irreducible objects of of $K_0^+[G]$. 

Proposition~\ref{order} implies that for two dominant weights $\mu,\la$ of $G$  we have $\mu\le \lambda$ if and only if there exists $\theta\in K_0^+[G]$ such that for all $n\in \N$ and all irreducible elements $[V_\nu]\in K_0^+[G]$, we have
$$[V_\mu]^n - [V_\nu] \in K_0^+[G]\Rightarrow [V_\lambda]^n\theta-[V_\nu] \in K_0^+[G].$$

{\bf Step 2.}
For every triple $\la,\mu,\nu$ of dominant weights  of $G$,
we have $\nu=\la+\mu$
if and only if $\la$ is the largest dominant weight such that $V_{\la}$ is an irreducible factor of $V_{\mu}\otimes V_{\nu}$. Therefore
Proposition~\ref{order} implies that the semigroup of dominant weights of $G$ can be reconstructed  from  the semiring $K_0^+[G]$.

{\bf Step 3.} The group of weights $X^*(T)$ of $G$ is the group completion of the semigroup of dominant weights.  The group of coweights of $G$, $X_*(T)$,  is given as the group of homomorphisms
$$X_*(T)=\Hom(X^*(T),\Z).$$
Note that there is a canonical isomorphism between $\Aut(X^*(T))$ and $\Aut(X_*(T))$.

{\bf Step 4.} We claim that $\al\in X^*(T)$ is a simple root if and only if it is a minimal weight of $T$ for which there exists a dominant weight $\lambda\in X^*(T)$ such that   $V_{2\lambda-\al}$  is an irreducible factor of $V_{\lambda}^{\otimes 2}$.

Indeed, for every simple root $\al$, we choose any dominant weight $\la$ such that 
$\langle\check{\al},\la\rangle=1$. Then it follows from Kumar's theorem \cite{Ku} that 
$2\lambda-\alpha = \lambda + s_{\alpha}(\lambda)$ is the highest weight of a factor of $V_\lambda^{\otimes 2}$. Conversely, if $V_{2\lambda-\beta}$ is an irreducible factor of $V_\lambda^{\otimes 2}$, then $2\lambda-\beta$ is a weight of $V_\lambda^{\otimes 2}$, thus
$\beta$ is of the form $a_1\alpha_1+\cdots+a_r\alpha_r$ with all $a_i$ non-negative integers. Hence any minimal such $\beta$ has to be a simple root.

By the Step 1, the set of simple roots can  therefore be reconstructed from the semiring $K_0^+[G]$. 

{\bf Step 5.} For each simple root $\al$ of $G$ the corresponding simple coroot $\check{\al}\in X_*$ can be characterized by the following condition:
for every dominant weight $\mu$ the pairing $\langle\check{\al},\mu\rangle$ is the unique element  $m\in \Z_{\geq 0}$ such that
$2\mu-m\al$ is dominant, but $2\mu-(m+1)\al$ is not dominant.  Indeed,
\[
\langle\check{\al},2\mu-m\al\rangle=2\langle\check{\al},\mu\rangle-m\langle \check{\al},\al\rangle= 2\langle\check{\al},\mu\rangle-2m
\]
is non-negative if and only if $m\leq \langle\check{\al},\mu\rangle$, while for every other simple root $\al'\neq\al$ of $G$ with a corresponding
simple coroot $\check{\al}'$, we have
\[
\langle\check{\al}',2\mu-m\al\rangle=2\langle\check{\al}',\mu\rangle-m\langle \check{\al}',\al\rangle\geq 2\langle\check{\al}',\mu\rangle\geq 0
\]
for all $m\geq 0$. Thus the set of simple coroots can also
be reconstructed from $K_0^+[G]$.

{\bf Step 6.} After having reconstructed  all simple coroots $\check{\al}$, we reconstruct all simple reflections $s_{\al}\in \Aut(X_*(T))$, hence the Weyl group $W\subset \Aut(X_*(T))$, as the subgroup generated by simple reflections. Next we reconstruct the set of all roots of $G$, as images of the simple roots under $W$,
and likewise for the coroots of $G$.
This completes the reconstruction of the whole root datum of $G$.
\end{proof}

\begin{Emp}
\begin{proof}[Proof of Theorem~\ref{kgroups}]
An isomorphism of semirings $\phi:K_0^+[G]\isom K_0^+[H]$ induces  a bijection between irreducible objects, hence a bijection between  dominant weights of $G$ and $H$, which we  denote by $\wt{\phi}$.

The proof of  Corollary~\ref{root datum} shows that $\wt{\phi}$ extends to the isomorphism between the root data of $G$ and $H$, hence it comes from an isomorphism of algebraic groups $\rho:H\isom G$.

We claim that $\rho^*:K_0^+[G]\isom K_0^+[H]$ is equal to $\phi$. It is enough to show that for each dominant weight $\la$ of $G$, we have $\phi([V_{\la}])=\rho^*([V_{\la}])$.  Both expressions, however, are equal to $[V_{\wt{\phi}(\la)}]$.

Conversely, if $\rho:H\to G$ is an isomorphism such that $\rho^*=\phi$, then for each dominant weight $\mu$ of $G$ we have $\rho^*([V_{\la}])=\phi([V_{\la}])=[V_{\wt{\phi}(\la)}]$, so $\rho$  induces the isomorphism $\wt{\phi}$ between the root data, hence $\rho$ is unique up to conjugation.

\end{proof}
\end{Emp}

\section{The Tannakian formalism}
In this section we are going to prove Theorem~\ref{main}. Throughout the section we will assume that  the hypotheses of Theorem~\ref{main} hold. For each irreducible representation $\om$ of $G$ we choose an irreducible representation $\rho_{\om}$
of $\Gm$ such that $[\rho_\om]=\phi([\om])$.

\begin{lemma}
\label{isom}
(a) Let $\om'$ and $\om''$ be two irreducible representations of $G$, and
let $\om'\otimes\om''\cong \oplus\om_i$ be a decomposition of their tensor product
into irreducibles. Then $\rho_{\om'}\otimes\rho_{\om''}\cong \oplus\rho_{\om_i}$.

(b)  If $\om$ is a trivial (one-dimensional) representation $\1$ of $G$, then $\rho_{\om}$ is a trivial representation of $\Gm$.

(c) The representation $\om$ is one-dimensional if and only if $\rho_{\om}$ is one-dimensional.

(d) For each irreducible representation $\om$ of $G$, we have $\rho_{\om^*}\cong(\rho_{\om})^*$.

(e)  Let $\om'$ and $\om''$ be two irreducible representations
of $G$ such that $\rho_{\om'}\cong\rho_{\om''}$. Then restrictions $\om'|_{G^{\der}}$ and $\om''|_{G^{\der}}$ are isomorphic.
\end{lemma}

\begin{proof}
(a) By hypothesis, we have $[\om']\cdot[\om'']=\sum_i [\om_i]$. Since $\phi$ is a homomorphism of semirings we conclude that
\[
[\rho_{\om'}\otimes\rho_{\om''}]=\phi([\om'])\cdot\phi([\om''])=\sum_i \phi([\om_i])=[\oplus\rho_{\om_i}].
\]
Since   $\rho_{\om'}$ and $\rho_{\om''}$ are irreducible, their tensor product  $\rho_{\om'}\otimes\rho_{\om''}$ is semisimple (see \cite[p.~88]{Ch}).
Therefore $\rho_{\om'}\otimes\rho_{\om''}\cong \oplus\rho_{\om_i}$.

(b) This follows from the observation that $\om=\1$ if and only if $\om\otimes\om\cong\om$.

(c) This follows from the observation that $\om$ is one-dimensional if and only if $\om\otimes\om$ is irreducible.

(d) Note that the representation $\om\otimes \om^*$ has a trivial subrepresentation $\1$. Therefore by (a) and (b), the representation
 $\rho_{\om^*}\otimes\rho_{\om}$ has a subrepresentation $\rho_{\1}\cong \1$. Since $\rho_{\om}$ and
 $\rho_{\om^*}$ are irreducible, this implies that $\rho_{\om^*}\cong(\rho_{\om})^*$.

 (e) If $\rho_{\om'}\cong\rho_{\om''}$, then   the tensor product $\rho_{\om'}\otimes(\rho_{\om''})^*\cong \rho_{\om'}\otimes\rho_{\om''^*}$ contains a
 subrepresentation $\1$. Using  (a) and (c), we conclude that the tensor product
 $\om'\otimes\om''^*$ has a one-dimensional subrepresentation $\xi$. Since $\om'$ and $\om''$ are irreducible, we conclude that
 $\om'\cong\om''\otimes\xi$, thus the restrictions $\om'|_{G^{\der}}$ and $\om''|_{G^{\der}}$ are isomorphic.
\end{proof}

\begin{Emp}
For every irreducible representation $\om$ of $G$ we denote by $z_{\om}$ its central character.
Let $Z$ be the center of $G$, and denote by  $\iota$ the embedding $\Gm\isom\Gm\times\{1\}\hookrightarrow\Gm\times Z(K)$.
\end{Emp}

\begin{lemma}
\label{density}
(a) There exists a unique homomorphism of semirings
$\widetilde{\phi}:K_0^+[G]\to K_0^+[\Gm\times Z(K)]$  such that
\begin{equation}
\label{center}
\widetilde{\phi}([\om])=[\rho_{\om}\boxtimes z_{\om}] \text{  for each irreducible } \om.
\end{equation}

(b) Moreover, $\widetilde{\phi}$ is injective, maps irreducibles to irreducibles and satisfies $\iota^*\circ \widetilde{\phi}=\phi$.

(c) Assume that there exists a homomorphism $\rho:\Gm\to G(K)$ such that $\rho^*=\phi$, and let
$\wt{\rho}:\Gm\times Z(K)\to G(K)$ be a homomorphism defined by $\wt{\rho}(\gm,z):=\rho(\gm)\cdot z$. Then $\wt{\rho}^*=\wt{\phi}$.
\end{lemma}

\begin{proof}
(a) Since the additive Grothendieck semigroup $K_0^+[G]$ is freely generated by irreducible elements $[\om]$, there exists a unique homomorphism of semigroups
$\widetilde{\phi}:K_0^+[G]\to K_0^+[\Gm\times Z(K)]$ which satisfies (\ref{center}). It remains to show  that for every two representations $\om',\om''$ of $G$ we have an equality
\begin{equation}
\label{mult}
\widetilde{\phi}([\om']\cdot [\om''])=\widetilde{\phi}([\om'])\cdot\widetilde{\phi}([\om'']).
\end{equation}
By the additivity of $\widetilde{\phi}$, we may assume that $\om'$ and $\om''$ are irreducible.
Let $\om'\otimes\om''\cong \oplus\om_i$ be a decomposition of their tensor product
into irreducibles. Then $[\om']\cdot [\om'']=\sum_i[\om_i]$, hence the left hand side of (\ref{mult}) is equal to
\[
\widetilde{\phi}(\sum_i[\om_i])=\sum_i\widetilde{\phi}([\om_i])=\sum_i[\rho_{\om_i}\boxtimes z_{\om_i}],
\]
while the right hand side of (\ref{mult}) is equal to
\[
[\rho_{\om'}\boxtimes z_{\om'}]\cdot[\rho_{\om''}\boxtimes z_{\om''}]=[(\rho_{\om}\otimes\rho_{\om''})\boxtimes z_{\om'} z_{\om''}].
\]
Since the central character of each $\om_i$ is equal to $z_{\om'}z_{\om''}$, equality (\ref{mult}) follows from Lemma~\ref{isom} (a).

(b) By construction, for each irreducible element $[\om]$, the element   $\widetilde{\phi}([\om])=[\rho_{\om}\boxtimes z_{\om}]$ is irreducible, and
$$\iota^*\widetilde{\phi}([\om])=\iota^*([\rho_{\om}\boxtimes z_{\om}])=[\rho_{\om}]=\phi([\om]).$$
This  implies that $\wt{\phi}$ maps irreducibles to irreducibles
and satisfies   $\iota^*\circ \widetilde{\phi}=\phi$.

Finally, since as additive semigroups,  $K_0^+[G]$ and $K_0^+[\Gm]$ are freely generated by irreducibles, in order to show that $\wt{\phi}$ is injective, it is enough to show that it is injective on irreducibles.

Let $\om'$ and $\om''$ are two irreducible representations of $G$ such that $\wt{\phi}([\om'])=\wt{\phi}([\om''])$.
Then $\rho_{\om'}\cong \rho_{\om''}$ and $z_{\om'}=z_{\om''}$. Using Lemma~\ref{isom} (e), we conclude that  $\om'|_{G^{\der}}\cong \om''|_{G^{\der}}$ and
$\om'|_{Z}=\om''|_Z$. Hence $\om'\cong\om''$, implying the injectivity.

(c) It is enough to show that $\wt{\rho}^*([\om])=\wt{\phi}([\om])$ when $[\om]$ is irreducible. Both expressions, however, are equal to $[\rho_{\om}\boxtimes z_{\om}]$.

\end{proof}

\begin{Emp}

\begin{proof}[Proof of  Theorem~\ref{main}]
First we will show the existence of $\rho$ under the assumption that  $\phi:K_0^+[G]\to K_0^+[\Gm]$ is injective.

Let $\cC$ be the full subcategory of $\Rep\Gm$ consisting of semisimple representations
$\tau\in \Rep\Gm$ such that $[\tau]=\phi([\om])$ for some $[\om]\in K_0^+[G]$. Since $\phi([\om])$ is irreducible for each irreducible $[\om]$,
$\cC$  is a semisimple abelian subcategory.  Since $\phi$ is a homomorphism of semirings, $\cC$ is a rigid tensor subcategory  of $\Rep\Gm$
(use Lemma~\ref{isom} (a)--(d)), hence a Tannakian category. Let $f:\cC\to\Vect$ be the forgetful functor, and let $H:=\Aut^{\otimes}(f)$  be the group of tensor automorphisms of $f$.

By the Tannakian formalism (see, for example, \cite[Thm 2.11]{DM})  $H$ is an affine group scheme, and $f$
induces an equivalence of tensor categories $\cC\isom\Rep H$.  Since $G$ is an algebraic group, the category $\Rep G$ has a tensor generator $\om$.
Then an element $\rho_{\om}\in\Rep\Gm$ such that $[\rho_{\om}]=\phi([\om])$ must be a tensor generator of $\cC\cong\Rep H$. This implies that
$H$ is an algebraic group (see \cite[Prop 2.20]{DM}).  Moreover, since $\cC\cong\Rep H$ is semisimple, the group $H$ is reductive
(see \cite[Prop 2.23]{DM}).

Every element of $\gm\in\Gm$ defines a tensor automorphism of $f$ over $K$.
Hence we get a group homomorphism $\pi:\Gm\to H(K)$ such that
$\pi^*:\Rep H\to\Rep\Gm$ is the inverse of the equivalence $f:\cC\isom\Rep H$.

By construction, the homomorphism $\phi:K_0^+[G]\to K_0^+[\Gm]$ decomposes as
$K_0^+[G]\overset{\phi'}{\longrightarrow}K_0^+[H] \overset{\pi^*}{\longrightarrow}K_0^+[\Gm]$, and the homomorphism $\phi'$ is surjective.
By our assumption, $\phi'$ is also injective, hence it is an isomorphism.  Since $G$ is connected, we conclude that $H$ is connected as well (use, for example,
\cite[Cor 2.22]{DM}). Therefore by Theorem~\ref{kgroups} there exists an isomorphism $\rho':H\isom G$ such that $\phi'=\rho'^*$.
Then the composition $\rho:=\rho'\circ\pi:\Gm\to G(K)$ satisfies $\rho^*=\phi'\circ\pi^*=\phi$.

To show the existence of $\rho$ in general, we consider the homomorphism of Grothendieck semirings $\wt{\phi}:K_0^+[G]\to K_0^+[\Gm\times Z(K)]$, considered in
Lemma~\ref{center}.

Then $\wt{\phi}$ is injective, so by the particular case  shown above, there exists a homomorphism
$\wt{\rho}:\Gm\times Z(K)\to G(K)$  such that $\wt{\rho}^*=\wt{\phi}$. Then the composition
$\rho:=\wt{\rho}\circ\iota:\Gm\to G(K)$ satisfies
$\rho^*=\iota^*\circ\wt{\phi}=\phi$.

Conversely, let $\rho:\Gm\to G(K)$ be a homomorphism such that $\rho^*=\phi$. To show that the Zariski closure of $\rho(\Gm)$ contains $G^{\der}$, it
suffice to show that the homomorphism $\wt{\rho}:\Gm\times Z(K)\to G(K)$ from Lemma~\ref{center} (c) has a Zariski closed image.

Let $H\subset G$ be the
Zariski closure of the image of $\wt{\rho}$, and denote by $i$ the inclusion $H\hookrightarrow G$. Then
$\wt{\rho}^*=\wt{\phi}:K_0^+[G]\to K_0^+[\Gm\times Z(K)]$ factors through $i^*: K_0^+[G]\to K_0^+[H]$.
In particular, $i^*$ is injective, and maps irreducibles to irreducibles. Then using Chevalley's theorem (\cite[Th.~5.1]{Bo}) or (\cite[Prop. 2.21]{DM}),
$i$ has to be an isomorphism.

Finally, to show that $\rho$ is unique up to  conjugation, it suffice to show that $\wt{\rho}:\Gm\times Z(K)\to G(K)$
is unique up to conjugation.  Thus we can replace $\rho$ by $\wt{\rho}$, and $\phi$ by $\wt{\phi}$,  thereby assuming that $\phi$ is injective.

Then, using  the notation of the existence part, the tensor functor $\rho^*:\Rep G\to\Rep\Gm$ decomposes as a composition
$\Rep G\overset{\psi}{\longrightarrow}\Rep H \overset{\pi^*}{\longrightarrow}\Rep \Gm$ of tensor functors. By the Tannakian formalism,
there exists a homomorphism $\rho':H\to G$ such that $\rho'^*=\psi$. Then $\rho$ is conjugate to the composition $\rho'\circ \pi$,
so it remains to show that the conjugacy class of $\rho'$ is uniquely defined.

We have seen that $\phi$ decomposes as  $K_0^+[G]\overset{\phi'}{\longrightarrow}K_0^+[H] \overset{\pi^*}{\longrightarrow}K_0^+[\Gm]$, therefore
$\rho'^*:K_0^+[G]\isom K_0^+[H]$ coincides with $\phi'$. Hence the uniqueness assertion for $\rho'$  follows from Theorem~\ref{kgroups}.
\end{proof}
\end{Emp}

\section{Two Corollaries}
In this section we are going to prove Corollaries ~\ref{cmain} and \ref{galrep}.

\begin{lemma}
\label{equiv}
 Assume that  the hypotheses of Corollary~\ref{cmain} hold.

(a)   There exists a unique
homomorphism of semirings $\phi:K_0^+[G]\to K_0^+[\Gm]$ such that $\phi([\om])=[\rho_{\om}]$ for each irreducible $\om$.

(b) Let $\rho:\Gm\to G(K)$ be a group homomorphism. Then equality (\ref{lift}) holds for $\rho$ if and only if
$\om\circ\rho\cong\rho_{\om}$ for all irreducible $\om$.
\end{lemma}

\begin{proof}
(a)  Since the semigroup $K_0^+[G]$ is freely generated by irreducible elements,
there exists a unique homomorphism of semigroups $\phi:K_0^+[G]\to K_0^+[\Gm]$ such that $\phi([\om])=[\rho_{\om}]$ for each irreducible $\om$.

It remains to show that  for every two representations $\om',\om''$ of $G$ we have an equality
$$\phi([\om']\cdot[\om''])=\phi([\om'])\cdot\phi([\om'']).$$
Since a semisimple representation is determined by its trace,
it is enough to show that
\begin{equation}
\label{tr2}
\Tr_{\phi([\om']\cdot[\om''])}(\gm) = \Tr_{\phi([\om'])}(\gm)\cdot \Tr_{\phi([\om''])}(\gm)
\end{equation}
for all $\gm\in\Gm$. First we observe that for all $\gm\in\Gm$ and  all $[\om]\in K_0^+[G]$ we have an equality
\begin{equation}
\label{tr}
\Tr_{\phi([\om])}(\gm) = \Tr_{[\om]}(f(\gm)).
\end{equation}
Indeed, by the additivity, it is enough to show (\ref{tr}) for $\om$  irreducible. In this case the assertion follows from equalities  $\phi([\om])=[\rho_{\om}]$ and (\ref{trace-condition}).

Using  (\ref{tr}), our desired equality (\ref{tr2}) can be written in the form
$$\Tr_{[\om']\cdot[\om'']}(f(\gm))=\Tr_{[\om']}(f(\gm))\cdot \Tr_{[\om'']}(f(\gm)).$$
Therefore it follows from the multiplicativity of the trace map
$\Tr:K_0^+[G]\to K[G]$.

(b) Since functions $\Tr_{\om}$ with $\om$ irreducible generate $K[c_G]$ as a $K$-vector space (see \cite[Theorem 6.1(a)]{St}), equality (\ref{lift}) is equivalent to the equality
\begin{equation}
\label{lift'}
\Tr_{\om}(\chi_G(\rho(\gm)))=\Tr_{\om}(f(\gm))
\end{equation}
for all $\gm\in\Gm$ and all irreducible $\om$.
Since the left side of (\ref{lift'}) equals $\Tr_{\om\circ\rho}(\gm)$ (see section \ref{cg} (b)), while the right hand side of (\ref{lift'})
equals $\Tr_{\rho_{\om}}(\gm)$ by (\ref{trace-condition}), equality (\ref{lift'}) is equivalent to the equality $\Tr_{\om\circ\rho}=\Tr_{\rho_{\om}}$
for all irreducible $\om$. But this is equivalent to the desired isomorphism $\om\circ\rho\cong\rho_{\om}$.
\end{proof}

\begin{Emp}
\begin{proof}[Proof of Corollary~\ref{cmain}]
By Lemma~\ref{equiv}  (a), there exists a unique homomorphism of semirings $\phi:K_0^+[G]\to K_0^+[\Gm]$
such that $\phi([\om])=[\rho_{\om}]$ for each irreducible $\om$. Then by Theorem~\ref{main} there exists a
homomorphism $\rho:\Gm\to G(K)$ such that $\rho^*=\phi$. In particular, we have $[\om\circ\rho]$ is equal
to $\phi([\om])=[\rho_{\om}]$ for each irreducible $\om$. Then by Lemma~\ref{equiv}  (b),  the equality (\ref{lift})
holds for $\rho$.

Conversely, let $\rho:\Gm\to G(K)$ be a homomorphism, satisfying (\ref{lift}). Then  by Lemma~\ref{equiv}  (b),
$\rho^*([\om])=[\om\circ\rho]$ is equal to $\phi([\om])=[\rho_{\om}]$ for each irreducible $\om$. Thus
$\rho^*:K_0^+[G]\to K_0^+[\Gm]$ is equal to $\phi$.
Therefore it follows from Theorem~\ref{main} that $\rho$ is unique up to conjugation,
and that the Zariski closure of $\rho(\Gm)$ contains $G^{\der}$.
\end{proof}
\end{Emp}

\begin{Emp}
\begin{proof}[Proof of Corollary~\ref{galrep}]
The argument is almost identical to that of Corollary~\ref{cmain}.

As in Lemma~\ref{equiv} (a), there exists a unique homomorphism of semirings $\phi:K_0^+[\hat G]\to K_0^+[\Gm_F]$ such that $\phi([\om])=[\rho_{\om}]$
for each irreducible $\om$. Indeed, arguing as in Lemma~\ref{equiv} (a) word for word, we reduce ourselves to the  equality (\ref{tr2}).  Moreover, by the Chebotarev density theorem, it is enough to show equality (\ref{tr2}) when $\gm=\Frob_v$ for almost all $v\in\Sigma$.

Then we reduce the problem to showing that
$$\Tr_{\phi([\om])}(\Frob_v)=\Tr_{[\om]}(f(v))$$
for all irreducible $[\om]$ and almost all $v\in\Sigma$. But the latter equality follows from equalities
$\phi([\om])=[\rho_{\om}]$ and (\ref{frob-condition}).

By Theorem~\ref{main}, there now exists a homomorphism $\rho:\Gm_F\to \hat G(\Qlb)$ such that  $\rho^*=\phi$.

We claim that for every representation $\om$ of $G$, the composition $\om\circ\rho$ is a semisimple $\ell$-adic representation.
By additivity, it is enough to show in the case when $\om$ is irreducible. However, in this case,
$$[\om\circ\rho]=\rho^*([\om])=[\rho_{\om}]$$
is irreducible,
hence $\om\circ\rho\cong \rho_{\om}$ is an irreducible $\ell$-adic representation.

Choosing $\om$ to be a faithful representation of $\hat G$, we conclude that $\rho$ is continuous and unramified almost everywhere.

Finally, arguing exactly as in Lemma~\ref{equiv} (b) (and using the isomorphisms $\om\circ\rho\cong \rho_{\om}$)
we conclude that $\rho$ satisfies the equality (\ref{lift1}).

Conversely, let  $\rho:\Gm_F\to \hat G(\Qlb)$ be a $\hat G$-valued $\ell$-adic homomorphism satisfying
(\ref{lift1}). Again arguing exactly as in Lemma~\ref{equiv} (b) and using the Chebotarev density theorem, we conclude that
$\rho^*([\om])=[\om\circ\rho]$ is equal to $\phi([\om])=[\rho_{\om}]$ for each irreducible $\om$. Thus
$\rho^*:K_0^+[\hat G]\to K_0^+[\Gm_F]$ is equal to $\phi$.

Therefore it follows from Theorem~\ref{main} that $\rho$ is unique up to conjugation,
and that the Zariski closure of $\rho(\Gm_F)$ contains $\hat G^{\der}$.
\end{proof}
\end{Emp}


\begin{thebibliography}{DM}

\bibitem[Bo]{Bo}
Borel, Armand:
Linear algebraic groups. Second edition. Graduate Texts in Mathematics, 126.
Springer-Verlag, New York, 1991.

\bibitem[Ch]{Ch}Chevalley, Claude:
Th\'eorie des groupes de Lie. Tome III. Th\'eor\`emes g\'en\'eraux sur les alg\`ebres de Lie.
Actualit\'es Sci.\ Ind.\ no.\ 1226. Hermann \& Cie, Paris, 1955.

\bibitem[DM]{DM}Deligne, Pierre; Milne, James S.:
Tannakian Categories, in \textit{Hodge cycles, motives, and Shimura varieties}. Lecture Notes in Mathematics 900.  Springer-Verlag, Berlin, 1982.

\bibitem[Ha]{Ha}Handelman, David: 
Representation rings as invariants for compact groups and limit ratio theorems for them. 
\textit{Internat.\ J.\ Math.} \textbf{4} (1993), no.\ 1, 59--88.

\bibitem[KV]{KV}
Kazhdan, David; Varshavsky, Yakov: On the cohomology of the moduli spaces of $F$-bundles: stable cuspidal
Deligne--Lusztig part, in preparation.



\bibitem[Ku]{Ku}
Kumar, Shrawan:
Proof of the Parthasarathy-Ranga Rao-Varadarajan conjecture.  \textit{Invent.\ Math.} \textbf{93} (1988),  no.\ 1, 117--130.

\bibitem[LP]{LP}
Larsen, M.; Pink, R.:
Determining representations from invariant dimensions.
\textit{Invent.\ Math.} \textbf{102}  (1990),  no.\ 2, 377--398.

\bibitem[Mc]{Mc}
McMullen, John R.:
On the dual object of a compact connected group. 
\textit{Math.\ Z.} \textbf{185} (1984), no. 4, 539--552. 

\bibitem[St]{St}Steinberg, Robert:
Regular elements of semisimple algebraic groups.
\textit{Inst.\ Hautes \'Etudes Sci.\ Publ.\ Math.\ No.\ 25} (1965), 49--80.

\end{thebibliography}
\end{document}